\documentclass[12pt,a4paper]{article}

\usepackage{amsfonts,amsmath,amsxtra,amsthm,latexsym}
\usepackage{amssymb}  
\usepackage{bbm} 

 \usepackage[utf8]{inputenc}

\numberwithin{equation}{section}

\newcommand\blfootnote[1]{%
  \begingroup
  \renewcommand\thefootnote{}\footnote{#1}%
  \addtocounter{footnote}{-1}%
  \endgroup
}

\newtheorem{thm}{Theorem}[section]
 \newtheorem{cor}[thm]{Corollary}
 \newtheorem{lem}[thm]{Lemma}
 \newtheorem{prop}[thm]{Proposition}
  \newtheorem{prob}{Problem}

 \theoremstyle{definition}
 \newtheorem{defin}{Definition}[section]
 \newtheorem{example}[thm]{Example} 
 \newtheorem{remar}{Remark}[section] 
 \newtheorem{algorithm}{Algorithm}[section]

\DeclareMathOperator{\im}{Im}
\DeclareMathOperator{\re}{Re}
\DeclareMathOperator{\Arc}{Arc}
\DeclareMathOperator{\supp}{supp}
\DeclareMathOperator{\dist}{dist}
\DeclareMathOperator*{\argmax}{arg\,max}

\def\RR{\mathbb R}
\def\CC{\mathbb C}
\def\NN{\mathbb N}
\def\ZZ{\mathbb Z}

\def\DD{\mathbb D}
\def\TT{\mathbb T}

\def\one{\mathbbm{1}}

\def\distPH{\dist_{\mathrm{PH}}}

\def\al{\alpha}

\def\vphi{\varphi}
\def\la{\lambda}
\def\th{\theta}
\def\de{\delta}
\def\ep{\epsilon}
\def\vep{\varepsilon}
\def\ga{\gamma}

\def\om{\omega}

\def\Om{\Omega}

\def\De{\Delta}

\def\A{\mathcal{A}}
\def\E{\mathbb{E}}
\def\P{\mathrm{Pr}}

\def\L{\mathcal{L}}

\def\bb{\mathfrak{b}}

\def\wt{\widetilde}

\def\wh{\widehat}

\def\ii{\mathrm{i}}
\def\dd{\mathrm{d}}
\def\ee{\mathrm{e}}

\def\Ar{\mathfrak{A}}
\def\FAr{\mathbb{A}}

\begin{document}

\title{Recovery of periodicities hidden in heavy-tailed noise}

\author{Illya M. Karabash, J\"urgen Prestin}

\date{}

\maketitle

\blfootnote{\noindent\textbf{Acknowledgments.}
 The authors are grateful to the anonymous referee for careful reading of the paper and a number of helpful remarks and questions.  The authors would like to thank Frank Filbir, Hrushikesh N. Mhaskar, and Boaz Nadler for valuable discussions. 
During various parts of this research, the authors were supported by the 
EU-financed projects
AMMODIT (''Approximation Methods for Molecular Modelling and Diagnosis Tools'',  
grant agreement MSCA-RISE-2014-645672-AMMODIT) and
EUMLS (''EU-Ukrainian Mathematicians for Life Sciences'', 
Marie Curie Actions - International Research Staff Exchange Scheme 
FP7-People-2011-IRSES, project number 295164);
the first named author was supported by the Alexander von Humboldt Foundation, by Hausdorff Research Institute for Mathematics, Bonn, and by
the project no.15-1vv{\textbackslash}19 
of Vasyl' Stus Donetsk National University.}

\begin{abstract}
We address a parametric joint detection-estimation problem 
for discrete signals of the form 
$x(t) = \sum_{n=1}^{N} \al_n \ee^{-\ii \la_n t } + \ep_t$, $t \in \NN$,
with an additive noise represented by
independent centered complex random variables $\ep_t$.
The distributions of $\ep_t$ are assumed to be unknown, but satisfying various sets of conditions.
We prove that in the case of a heavy-tailed noise it is possible to construct asymptotically
strongly consistent estimators for the unknown parameters of the signal, i.e.,
frequencies $\la_n$, their number $N$, and complex coefficients  $\al_n$.
For example, one of considered classes of noise is the following:
$\ep_t$ are independent identically distributed  random
variables with $\E (\ep_t) = 0$ and $\E (|\ep_t| \ln |\ep_t|) <  \infty$.
The construction of estimators is based on 
detection of singularities of anti-derivatives for 
$Z$-transforms and on a two-level selection procedure for special discretized versions
of superlevel sets. The consistency proof relies on the convergence theory for random Fourier series. We discuss also decaying signals and the case of infinite number of frequencies.
\end{abstract}


MSC-classes: 
94A12, 
42A70,  
42A61, 
62F12, 
42A24

\tableofcontents  

\section{Introduction}
\label{s:intro}


Consider a signal modeled by the complex time series of the form
\begin{equation} \label{e x}
x(t) = \sum_{n=1}^{N} \al_n \ee^{-\ii \la_n t } + \ep_t , \qquad t \in \NN ,
\end{equation}
consisting of a finite number $N \in \ZZ^+$ of periodicities  $\al_n \ee^{-\ii \la_n t }$
with distinct real frequencies $\la_n \in [-\pi,\pi)$ and complex amplitudes $\al_n \neq 0$.
The signal is corrupted by a random noise $( \ep_t )_{t=1}^\infty$. 
The problem of finding, or estimation of the unknown parameters $N$, $\la_n$, and $\al_n$
from a finite number of samples $\left( x(t) \right)_{t=1}^{m} $ is a fundamental problem in
signal processing with a number of applications ranging from speech recognition and direction
finding in array antennas to
astrophysics, medicine, and economics (see \cite{B87,ZS01,SM05,FMP12} and references
therein).

When the number $N$ of periodicities is unknown, the number of spectral data is not bounded by a known finite number.
So the numerical recovery can be expected only via a convergent process with theoretically infinite number of steps. 
On each step the process has to use only a finite part 
of the signal and should give certain approximations to unknown parameters 
$N$, $z_n:=e^{\ii \la_n}$, and $\al_n$.  These approximations  $\wh N = \wh N (m)$,
$\wh z_n = \wh z_n (m)$, $\wh \al = \wh \al (m)$ are called estimators and, in the case of random noise, are functions of 
random variables $x(t)$, $t =1, \dots, m$.
An estimator is called (asymptotically) \emph{consistent} if it converges with $m \to +\infty$ to the corresponding parameter 
in a certain probabilistic sense. 

This paper is aimed on the case when the random noise variables $\ep_t$ are allowed to have \emph{heavy-tailed distributions} and studies  \emph{strong consistency} of estimators, which corresponds to almost sure convergence. The interest to the heavy-tailed $\ep_t$  is stimulated by the paper of Zhou and Sornette \cite{ZS01}, 
which numerically investigates the case of $\ep_t$ with infinite variances $\E (|\ep_t |^2)$ 
and lists a number of applications.
The latter includes vortices in freely decaying 2-D turbulence,
ion-signature precursors of earthquakes, 
and the price dynamics of speculative bubbles preceding
financial collapses. 

The most difficult part of the problem is usually the detection of the unknown number $N$ of periodicities.
The presently available analytic proofs of consistency 
require the assumption  $\E (|\ep_t|^4) < \infty$.
However, the numerical experiments of \cite{ZS01} suggest 
that consistent estimation is possible even in 
the case when variances $\E (|\ep_t |^2)$ are infinite.
 
The goal of this paper is to study analytically the case of independent identically distributed (i.i.d.) random variables 
$\ep_t$ with infinite variances and to prove that consistent estimation 
of all unknown parameters is possible under fairly mild assumptions on the distribution tails of
$\ep_t$.  

\textbf{Notation}. The following sets of real and complex numbers are used:
open half-lines $\RR_\pm = \{ x \in \RR: \pm x >0 \}$,
nonnegative and nonpositive integers $\ZZ^\pm = \pm \NN \cup \{ 0\}$,
the unit circle $\TT =\{z \in \CC : |z|=1\}$, and the unit disc $\DD = \{z \in \CC : |z|<1\} $.
When we consider the real interval $[\th_1,\th_2]$, or arcs 
\[
\Arc [\th_1,\th_2] := \{ \ee^{\ii \th } \ : \ \th_1 \le \th \le \th_2 \}  \subset \TT, 
\]
we assume that $\th_1 \le \th_2$.
%
For $\zeta \in \CC$ and complex sets $G_{1,2}$,
\[
G_1 + G_2 := \{ z_1 +z_2 \ : \ z_1 \in G_1, z_2 \in G_2 \} ,
\qquad \zeta G_1 := \{ \zeta z \ : \ z \in G_1 \}.
\]

Let $\one := (u_t)_{t=1}^{+\infty} \in \CC^\NN$ with $u_t =1$ for all $t \in \NN$.
The notation $ \lceil s \rceil $ stands for the ceiling function, i.e., the smallest integer not less than $s$;
$\ln^\de s$ is $(\ln s)^\de$.   
 By $\L^p $ and $\| \cdot \|_p$ we denote the standard Lebesgue spaces of 
complex-valued functions and the
corresponding norms, respectively. 

\section{Main results and employed techniques}


The problem of hidden sinusoids has a long history with periodicities in orbital data 
detected numerically as early as 1754 \cite{B87}.
The classical statement of the problem involves the discrete signal of form
$
x(t) = \sum_{n=1}^{N} \al_n \ee^{-\ii \la_n t}
$
equidistantly  sampled over some finite set of $t$. Initially, the number $N$ of  frequencies
was supposed to be known and finite.
Perhaps, the oldest analytic algorithm dates back to G.R. de Prony
(see e.g. \cite{PT14}). The Prony's method and the essentially equivalent annihilating
filter method have a number of advantages. They
utilize the minimal possible number of observations,
work in the case of damped oscillations, and give exact result in
the absence of noise.
Besides the assumption that $N$ is given, it is usually assumed that
the essential drawback of Prony's method is poor performance when data are too noisy. 

For the case of \emph{a priori} knowledge of a bound on
the number $N$ of frequencies, a number of deterministic and statistical methods
were developed to deal with signals of type (\ref{e x})
involving various models of uncertain or random noise $\ep_t$.
Let us mention various modifications and developments of Prony's and ESPRIT methods
\cite{RK89,PT10,PT13,PP13,PT14}
(in deterministic settings the ESPRIT method of \cite{PRK85,RK89}
has some common ideas with Prony's methods, see discussion in \cite{PT14}),
the approach of Jones, Njåstad, and Saff employing 
Wiener-Levinson filters and Szegö polynomials
\cite{JNS90,PS92,P96,NW97,A07,FMP12}, and a variety of statistical methods
\cite{B87,RK89,SMFS89,QH01,SM05}
(for recent developments on the multivariate case see \cite{KF13,PT13_mult,KPRvdO15,KMvdO16}).

The problem of determination of the number $N$ of sinusoids attracted
attention in mid-80s \cite{B87,SM05} (it was placed in
the list of open problems at the end of \cite{B87}).
Several statistical methods have been developed to obtain asymptotically consistent
estimators for $N$ and the parameters of sinusoids (see
\cite{F88,KH94,QH01,PdLvH07,KN09,NK11,KF13}
and the reviews in \cite{SM05,NK11}).

While most of studies work
with a white Gaussian noise with a known or unknown standard deviation,
in many practical situations the distribution of noise is unknown. Some
steps for lifting of Gaussian and independency assumptions have been done.
Kavalieris and Hannan \cite{KH94} deal with a `colored` autoregressive noise with
innovations sequence $\vep (t)$ satisfying $\E (\vep(t)^4) < \infty$ and prove that 
their estimator provides strongly consistent detection of the number $N$ of sinusoids.
The performance of periodogram and information criteria methods in presence of 
non-Gaussian $\ep_t$ was studied numerically by Zhou and Sornette \cite{ZS01} and 
Nadler and Kontorovich \cite{NK11}.
Under the assumption that $N$ is known, the question of probability estimates for localization of $\la_n$ in the presence of a non-Gaussian noise was raised recently in \cite{FMP12} in connection with a modified method of orthogonal polynomials and the estimation techniques of \cite{MP00}.

Presently there exists a gap between assumptions of analytically proved consistency 
results and numerical evidences of the fact that many 
of statistical methods perform well under weaker 
restrictions on noise. 

The main goal of the present paper is to fill this gap
and to certify analytically that recovery of the unknown parameters   
is possible in the case of a heavy-tailed noise.
We prove rigorously that strongly consistent estimation of the parameters
$N \in \ZZ^+$, $\la_n \in [-\pi,\pi)$, and $\al_n \in \CC \setminus \{0\}$ 
of the signal (\ref{e x})
is possible in presence of complex random noise sequences 
$( \ep_t )_{t=1}^{\infty}$ belonging to the following classes:
\begin{itemize}
\item[(N1)] the random variables $\ep_t$ are i.i.d. with $\E (\ep_t) = 0$
and $\E (|\ep_t| \log |\ep_t|) < \infty$ for all $t$;
\item[(N2)] $\ep_t$ are i.i.d. symmetric and 
satisfy $\E \left( |\ep_t| \log \log (|\ep_t|+\ee) \right) < \infty$.
\end{itemize}

Besides of heavy-tailed $\ep_t$, we consider also the case when variances of $\ep_t$ are finite, 
but not uniformly bounded in $t$. In this case, the consistency is proved under the assumption of sub-linear 
growth of $\E (|\ep_t|^2)$. More precisely, we will consider the following class of noise sequences:
\begin{itemize}
\item[(N3)] $\ep_t$ are independent, symmetric, and have finite variances  $\E (|\ep_t|^2) $ 
satisfying \linebreak
$\E (|\ep_t |^2) = O (t^{\nu})$ as $t \to +\infty$ for certain $\nu < 1$.
\end{itemize}

While our approach has some common features with 
the classical methods (see e.g. \cite{SM05}) involving detection of peaks of
periodograms and Welch temporal windows, 
it has the following novelties:
\begin{itemize}
\item We are aimed at the detection of singularities
of the \emph{anti-derivative} $S_x $ of the $Z$-transform of $\left( x(t) \right)_{t=1}^{+\infty}$.
It is easy to see that the anti-derivative of the deterministic signal
$y(t)=\sum_{n=1}^{N} \al_n \ee^{-\ii \la_n t }$ defined by 
\[
\text{$S_y (\ee^{\ii \th}):=
\sum_{t=1}^{+\infty} \ee^{\ii t \th} y(t) / t $}
\]
 has logarithmic singularities
at the points $\ee^{\ii \la_n}$ of the unit circle $\TT = \{\ee^{\ii \th} : \th \in \RR \}$.

\item Strong consistency is proved with the use of uniform convergence theorems for 
\[
\text{the random
 series $S_\ep (\ee^{\ii \th}):=
\sum_{t=1}^{+\infty} \ee^{\ii t \th} \ep_t / t $ corresponding to the noise part 
of $S_x$.
}
\]
This allows us to avoid the estimation of distribution tails  
of $\L^\infty$-norms of random trigonometric polynomials 
generated by $\ep_t$,
which is not adequately studied in the non-subexponential case.
One of the main points of this paper is to show that the results of Kahane \cite{K85} and 
Cuzick and Lai \cite{CL80} on uniform 
convergence of random Fourier series are powerful enough for estimation
of parameters in the presence of a heavy-tailed noise.

\item The estimators for $\la_n$ are generated not by peaks of partial sums of $S_x$, but by
discretized versions of superlevel sets associated with special
temporal windows, i.e., with partial sums 
\begin{equation} \label{e SAxm}
S^A_{x,m} (\ee^{\ii \th}) := \sum_{t=1}^{m}  \frac{a_{m,t}}{t} x_t \ee^{\ii t \th} 
\end{equation}
defined by a special summation matrix $A= (a_{m,t})_{m,t=1}^{+\infty}$, see 
Sections \ref{ss:model_wind} 
and  Proposition \ref{ss:decrKer}.

\item While the type of summation is not important for consistent localization 
of isolated frequencies
(see Section \ref{ss:loc}), it
becomes essential  for asymptotically consistent detection of frequencies and
estimation of their number (Sections \ref{s stab}). 
Indeed, the noise part $S^A_\ep$ and possible side lobes of summation kernels
will produce peaks and superlevel sets 
of $|S^A_x|$ that do not directly correspond
to frequencies $\la_n$,
but rather lie nearby and accompany the superlevel sets containing the frequencies.

\item To filter out such \emph{side superlevel sets} we develop 
in Section \ref{s stab}
a two-level selection approach and employ special kernels with 'almost-monotonicity' property
(see Section \ref{ss:decrKer}).
\end{itemize}

\section{Consistent localization of frequencies}
\label{s:loc}

\subsection{Main setting and temporal windows}
\label{ss:model_wind}

Consider a discrete signal $x $ of the form (\ref{e x}) with a random noise 
$( \ep_t )_{t=1}^{+\infty}$
consisting of independent random variables  $\ep_t$ 
defined on a complete probability space $(\Om,\A,\P)$.
We assume that 
the set $\{ \la_n \}_{n=1}^{N }$ consists of a finite
number $N \in \ZZ^+$ of distinct real frequencies $\la_n \in [-\pi,\pi)$.
With the deterministic part $y$ of the signal we 
associate the complex Borel measure 
\[
\mu = \sum_{n=1}^N \al_n \de (z- \ee^{\ii \la_n})
\]
on the unit circle $\TT $ consisting of (complex) point masses 
$\al_n \in \CC \setminus \{0\}$ placed at points 
\[
z_n = \ee^{\ii \la_n} .
\]
So every frequency $\la_n$ corresponds to  the $\de$-function term 
$\al_n \de (z- z_n )$ and 
the  \emph{support of the measure} is given by 
\[ \text{
$\supp \mu = \{ z_n \}_{n=1}^N $ \quad  if $N>0$ \quad and  \quad 
$\supp \mu = \varnothing$ \quad if $N=0$.
} \]

Then the signal $x$ can be written in the form
\[
x(t) =y(t)+\ep_t , \qquad \text{ where }  y(t)=\int_\TT z^{-t} \dd \mu .
\]

The problem under consideration 
is to develop consistent estimators that recover 
either the support of $\mu$,
or completely the measure $\mu$  from one sample 
of the random vector 
$x=( x(t) )_{t=1}^{+\infty} = ( x (t,\om) )_{t=1}^{+\infty}$, $\om \in \Om$,   
under certain additional assumptions on random noise components $\ep_t (\om)$.
Our goal is to construct \emph{strongly consistent} \emph{estimators} 
$\wh N = \wh N (m,\om)$, $\wh z_n =\wh z_n (m,\om) $, and 
$\wh \al_n =\wh \al_n (m,\om)$, which, by definition, depend 
only on the finite parts $( x(t) )_{t=1}^{m} $ of the signal and 
have to  converge with $m \to +\infty$ to true spectral data 
$ N $, $z_n $, and 
$\al_n $ on 
a certain almost sure (a.s.) event of the probability space.

With an arbitrary signal $u = \left( u(t) \right)_{t=1}^{+\infty} $, 
we associate  its weighted modification $W= \left( u(t)/t \right)_{t=1}^{+\infty} $ 
and the geophysical $Z$-transform $ S_u (z) = \sum_{t=1}^{+\infty} z^t u(t) / t $ of $W$
defined as a function of a complex variable $z$ on the set of its convergence. 
The function $S_u$ is an anti-derivative of the
$Z$-transform $\sum_{t=0}^{+\infty} z^t u(t+1) $ 
of the shifted signal $\left(  u(t+1) \right)_{t=0}^{+\infty} $.

For the unilateral Fourier series $S_u (\ee^{\ii \th})$ we will use various
summation procedures defined via infinite summation matrices $A = (a_{m,t})_{m,t=1}^{+\infty}$. 
Partial sums associated with the matrix $A$ are given by
$
S_{u,m}^A (\ee^{\ii \th}) := \sum_{t=1}^{+\infty} a_{m,t} \frac{u_t}{t} \ee^{\ii t \th} .
$

It will be supposed that $(a_{m,t})$ satisfy the following conditions:
\begin{gather}
0< a_{m,1}  \le  1   \quad \text{ and }   a_{m,t}  \text{ is non-increasing in $t$ for every } m \in \NN ,
\label{e a 1} \\
 a_{m,t} \ge 0 \text{ if  } t \le m , \qquad   a_{m,t} = 0 \text{ if  } t > m ,
\label{e a 2}\\
 a_{m,t}  \to  1   \   \text{ as } m\to +\infty   \qquad \quad \text{ for every } t \in \NN ,
\label{e a 3} \\
 a_{m,t} \text{ is non-decreasing in $m$ for every } t \in \NN .
\label{e a 4}
\end{gather}

Under these conditions $A$-summation is \emph{regular}. 
That is, for any strongly convergent  series $V_0 = \sum_{t=1}^{+\infty} V_t$ 
(with $V_t$ from a certain Banach space),
the sequence of $A$-partial sums $\sum_{t=1}^m a_{m,t} V_t$ is also convergent to $V_0$,
which follows from the Toeplitz regularity test, see e.g. \cite[Theorem 3.2.2]{H49}.

We will need the following simple corollary of the regularity of $A$-summation.
\begin{lem} \label{l uni y}
Let $\mu $ be a complex Borel measure 
supported in a finite number of points 
of $\TT$. Then 
$S^A_{y,m} (\ee^{\ii \th})$ converge uniformly on 
each closed arc of $\TT$ disjoint with $\supp \mu$.
\end{lem}

\begin{proof}
Note that
$
S_{y} (\ee^{\ii \th}) = - \sum_{n=1}^{N} \al_n 
\log (1-\ee^{\ii (\th -\la_n)}) \in \L^1 (\TT),
$
where the sum converges with respect to the norm of $\L^1 (\TT)$.
Since $S_{y} (\ee^{\ii \th})$ is smooth on $\TT\setminus \supp \mu$,
the partial sums $\sum_{t=1}^n \ee^{\ii t \th} y(t)/ t $ converge uniformly 
on each closed arc of $\TT$ disjoint with $\supp \mu$.
The regularity of the $A$-summation completes the proof.
\end{proof}

In the next subsection, we will use the notion of set convergence following, e.g., \cite{RW}. Namely,
for $z \in \CC$, $G \subset \CC$, $G_{1,2} \subset \CC$, put
\begin{gather*}
\dist (z,G) := \inf_{\zeta \in G} |z-\zeta| \ \text{ if } G \neq \varnothing,  \quad
\ \text{ and }  \dist (z,G)= +\infty \ \text{ if } G = \varnothing  
\\
\distPH (G_1,G_2) := \sup_{z \in \CC} | \dist (z,G_1) - \dist (z,G_2) |  \ \text{ if } G_1 \cup G_2 \neq \varnothing , \\
\text{ and }  \distPH (G_1,G_2) := 0  \ \text{ if } G_1 = G_2 = \varnothing .
\end{gather*}
If both of the sets $G_{1,2}$ are nonempty and closed,  
$\distPH$ is \emph{the Pompeiu-Hausdorff distance} and can be defined by an 
alternative formula 
\begin{equation} \label{e:distHP2}
\distPH (G_1, G_2) = \inf \{ \eta \ge 0 \ : 
\ G_1 \subset G_2 + \eta \overline{\DD} , \ G_2 \subset G_1 + \eta \overline{\DD} \} \ , 
\end{equation}
where $\overline{\DD}$ is the closed unit disc in $\CC$, see \cite[formula 4(5)]{RW}.

Having a sequence $\left( L_m \right)_{m=1}^{+\infty}$ of closed subsets of $\TT$, we say that
$\left( L_m \right)_{m=1}^{+\infty}$ \emph{converges} to a closed set $G$ 
(and write $\lim L_m = G$) 
if \quad $\lim\limits_{m\to +\infty} \distPH (L_m,G) = 0$.
Since only the subsets of the bounded set $\TT$ are considered, 
this convergence is a restriction of the Painlev\'{e}-Kuratowski convergence 
on closed subsets of $\TT$, and coincides with the convergence 
with respect to  Pompeiu-Hausdorff distance whenever $L_m$ and  $G$
are nonempty.

\subsection{Localization by discrete superlevel arcs}
\label{ss:loc}

In this subsection, let us fix  an arbitrary summation matrix 
$A$ satisfying (\ref{e a 1})-(\ref{e a 4}).
The  kernel $S_{\one,m}^A $ associated with 
$A$-summation of the anti-derivative is defined by 
\[
S_{\one,m}^A (\ee^{\ii \th}) := \sum_{t=1}^{m} \frac{a_{m,t}}{t} \ee^{\ii t \th} .
\]

The goal of this section is to produce strongly consistent estimators of the support of the measure $\mu$,
i.e., to produce random subsets $\wh L_m $ of $\TT$ that a.s. converge to $\supp \mu$
as $m \to +\infty$. These estimators will have an additional property that, with probability 1,
$\supp \mu \subset \wh L_m$ for $m$ large enough.  This property will be crucial for
the construction of estimators for the parameters of the signal in the next section.

The following discrete versions of superlevel sets serve as 
building blocks for $\wh L_m$.

\begin{defin} \label{d cluster}
We define a \emph{discrete superlevel arc} of level $h \in \RR$
and grid order $J$ for the partial $A$-sum $S_{u,m}^A$ 
(in short, $(h,J)$-arc for $S_{u,m}^A$)
as a closed arc in $\TT$ of the form
$\displaystyle \Arc \left[ \frac{2\pi j_1}{J} , \frac{2 \pi j_2}{J} \right]$ with
$j_{1,2} \in \ZZ$ such that 
\[
\left| S_{u,m}^A \left( \exp \left(\ii \frac{2 \pi  j}{J} \right)\right)  \right|
\ge h \qquad
 \text{ for all } j \in \ZZ \cap [j_1,j_2] .
\]
\end{defin}

Levels and orders of grids will depend on the length  $m$ of the signal 
and have to be connected with the matrix $A$. This is done in the following
way. 

Since $S^A_{\one,m} (1) = \sum_{t=1}^m a_{m,t}/t$ goes to $+\infty$ 
as $m \to +\infty$, there exists an auxiliary sequence $\left( H_m \right)_{m=1}^{+\infty} $ of 
real numbers
such that 
\begin{equation} \label{e Hk1}
H_m < S^A_{\one,m} (1)  \ \text{ for all $m$, \qquad  and }  \ \lim H_m = + \infty .
\end{equation}
When such a sequence  $\left( H_m\right)_{m=1}^{+\infty} $ is chosen, let us fix also
a sequence $\left( h_m \right)_{m=1}^{+\infty} $ of real numbers satisfying  
\begin{equation} \label{e hk}
\lim h_m = \lim \frac{H_m}{h_m} = +\infty ,
\end{equation}
and a sequence $\left( J_m \right)_{m = 1}^{+\infty} $ of natural numbers so that
\begin{equation} \label{e:Jm}
|S^A_{\one,m} (\ee^{\ii \th})| \ge H_m \text{ for all } \th \in [-2 \pi/J_m , 2 \pi/J_m] .
\end{equation}
The latter is always possible due to the continuity of $S_{\one,m}^A (\ee^{\ii \th})$.
Examples will be given below.

Let us note that (\ref{e:Jm}), $\lim H_m = +\infty$, and Lemma \ref{l uni y} imply 
\begin{equation} \label{e n inf d 0}
\lim J_m = +\infty .
\end{equation}

For each $m \in \NN$, we denote by
\begin{equation} \label{e Lk}
\text{ $\wh L_m := \wh L_m (\om)$ 
the union of all $(h_m, J_m)$-arcs for $S_{x,m}^A$,}
\end{equation}
which is assumed to be the empty set 
in the case when such superlevel arcs do not exist.

\begin{thm}[consistent localization] \label{t conver set}
Assume  that at least one of conditions (N1)-(N3) of Section \ref{s:intro}
is fulfilled for the random noise $\ep$. 
Then almost surely 
\[
\text{the sets $\wh L_m $ converge to $\supp \mu$ and contain $\supp \mu$ for $m$ large enough}.
\]
\end{thm}

The proof is given in the next subsection. Let us provide an example of possible 
choices of $H_m$, $h_m$, and $J_m$ for the case of the Dirichlet summation.

\begin{example} \label{ex:Dir}
The matrix $A$ associated with the Dirichlet summation is defined by
\begin{equation} \label{e AD}
a_{m,t} =  1 \text{ if } t \le m, \qquad a_{m,t} = 0 \quad  \text{ if } t > m .
\end{equation}
If $|\th| \le \frac{\pi}{4m}$, one has
$
| S^A_{\one,m} (\ee^{\ii \th}) |\ge 
\left| \sum_{t=1}^{m} \frac{\cos (t \th)}{t} \right| \ge 
\frac{1}{\sqrt{2}} \sum_{t=1}^{m} 
\frac{1}{t} \ge \frac{\ln (m+1)}{\sqrt{2}} .
$
So (\ref{e Hk1}), (\ref{e hk}), and (\ref{e:Jm}) are fulfilled if 
\begin{equation} \label{e:HJh_AD}
H_m =\frac{\ln (m+1)}{\sqrt{2}} , \qquad J_m = 8m , \quad \text{ and } \quad 
h_m = c_0 + c_1 \ln^{1-\de} m  ,
\end{equation}
where $\de \in (0,1)$, $c_0 \in \RR$,
and $c_1 \in \RR_+$ are constants.
\end{example}

\subsection{Proof of Theorem \ref{t conver set}}
\label{ss conver set}

Recall that $S_\ep (\ee^{\ii \th}):=
\sum_{t=1}^{+\infty} \ee^{\ii t \th} \ep_t / t $ is the random
Fourier series associated with the noise part of $S_x$.  

\begin{prop} \label{p det conver}
Assume  that $\{ \ep_t \}_{t=1}^{+\infty}$ 
satisfies the following condition:
\item[(N0)] the set $B_1$ of $\om \in \Om$ such that 
$\sup_{m\in \NN} \|S_{\ep ,m}^A (\ee^{\ii \th})\|_\infty < \infty $ 
is an almost sure event.

Then for every $\om \in B_1$ the following statements hold:
\item[(i)] each point $z_n=\ee^{\ii \la_n}$ of $\supp \mu$ 
is contained in $ \wh L_m $ for $m$ large enough,
\item[(ii)] each closed arc of $\TT$ disjoint with 
$\supp \mu$ is also disjoint with $\wh L_m $ for $m$ large enough.
\end{prop}

\begin{proof}
\textbf{(i)}
Let 
\begin{equation} \label{e:u_mu_n}
\text{$v(t) = \al_n \ee^{- \ii \la_n t}$, \ $t \in \NN$, \qquad and \qquad $\mu_n := \mu - \al_n \de (z-z_n)$.}
\end{equation}
Then there exists $\ga>0$ such that 
$\Arc [\la_n -\ga, \la_n+ \ga]$ is disjoint with 
$\supp \mu_n$. 
By Lemma \ref{l uni y}, $S^A_{y-v, m} (\ee^{\ii \th})$ converge uniformly on
$\Arc [\la_n -\ga, \la_n + \ga]$ as $m \to +\infty$. 
This, $\om \in B_1$, and the definition of the a.s. event $B_1$ imply
\begin{equation} \label{e C1 sup max}
C_1 := \sup_{m \in \NN}  \ \max_{|\th -\la_n| \le \ga} |S^A_{x-v,m} (\ee^{\ii \th})| \ < \ +\infty .
\end{equation}
It follows from $\lim J_m = +\infty$ that there exists $m_1$ such that
$2 \pi / J_m \le \ga$ for $m \ge m_1$. For such $m$, 
formulas (\ref{e:Jm}) and (\ref{e Hk1}) imply for $\th$ satisfying 
$| \th - \la_n| \le 2 \pi / J_m $ that
\[
|S^A_{x,m} (\ee^{\ii \th})| \ge |S^A_{v,m} (\ee^{\ii \th})| - 
|S^A_{x-v,m} (\ee^{\ii \th})| \ge
 | \al_n S^A_{\one,m} (\ee^{\ii (\th-\la_n)})| - C_1\ge | \al_n | H_m  - C_1.
\]
By (\ref{e Hk1}) and (\ref{e hk}), there exists $m_2 \ge m_1$ such that
$|S^A_{x,m} (\ee^{\ii \th}) | \ge h_m $ for  $ m \ge m_2 $ and
$| \th - \la_n| \le 2 \pi / J_m $.

Thus, for $m \ge m_2$, the point $\ee^{\ii \la_n}$ 
is contained in one of $(h_m, J_m )$-arcs for $S^A_{x,m}$, 
and so  $\ee^{\ii \la_n}$ is contained in $\wh L_m $.

\textbf{(ii)}
Consider an arc $\Arc [\th_1,\th_2]$ disjoint with $\supp \mu$.
By Lemma  \ref{l uni y} and the definition of the event $B_1$,
\[
C_2 := \sup_{m \in \NN}  \ \max_{ \th_1 \le \th \le \th_2 } \
|S^A_{x,m} (\ee^{\ii \th})| < \infty .
\]
This implies that $\Arc [\th_1,\th_2]$ is disjoint with $\wh L_m$ as soon as $C_2 < h_m$.
Condition (\ref{e hk}) completes the proof.
\end{proof}

Consider one more noise class defined by the following assumptions:
\begin{itemize}
\item[(N4)] $\ep_t$ are independent, symmetric, and have finite variances  $\E (|\ep_t|^2) $ 
satisfying the condition $\sum_{k=1}^{+\infty} \tau_k < +\infty$, where 
$\tau_k := 2^{k/2} 
\left( 
\sum\limits_{t=2^{2^k}}^{-1+2^{2^{(k+1)}}}  \frac{\E (|\ep_{t+1}|^2)}{(t+1)^2}
\right)^{1/2}$.
\end{itemize}

\begin{lem} \label{l:NtoN0}
Each of conditions (N1)-(N4) implies (N0).
\end{lem}

\begin{proof}
The implications (N1)$\Rightarrow$(N0) and (N2)$ \Rightarrow $(N0)
follow from \cite{CL80}. Indeed, \cite{CL80} implies that, for each of the classes (N1) and (N2),
usual partial sums $\sum_{t=1}^m \ep_t \ee^{\ii t \th}/ t $ a.s. 
converge uniformly on $\TT$ .
The regularity of the $A$-summation process implies that $S_{\ep,m}^A$ a.s. 
converge uniformly on $\TT$, and so are a.s. uniformly bounded.

Similarly, (N4) $\Rightarrow$ (N0) follows from the results of \cite[Section 7.2]{K85}.

To show (N3) $\Rightarrow$ (N4), it is enough to notice that, under the assumptions  
$\nu \in (0,1)$ and 
$\E (|\ep_t|^2) < C t^\nu$, we have  
\begin{multline*}
\tau_k \le C^{1/2} 2^{k/2}  
\left( 
\sum\limits_{t=2^{2^k}}^{-1+2^{2^{(k+1)}}}  (t+1)^{(\nu-2)} 
\right)^{1/2} 
 \\  \le 
C^{1/2} (1-\nu)^{1/2} 2^{k/2} \left( 
2^{(\nu - 1) 2^k} - 2^{(\nu -1) 2^{(k+1)}} \right)^{1/2}
\le C^{1/2} (1-\nu)^{1/2} 2^{k/2} 2^{(\nu - 1) 2^{k-1}} ,
\end{multline*}
and so $\sum_{k=1}^{+\infty} \tau_k < +\infty$.
\end{proof}

Now, Theorem \ref{t conver set} easily follows from 
(\ref{e:distHP2}), Proposition \ref{p det conver}, and Lemma \ref{l:NtoN0} . 

\begin{remar}
The above proof shows that condition (N3) can be seen as a transparent particular case 
of the cumbersome looking condition (N4). We do not expect that under assumption of i.i.d. and $\E (|\ep_t|^2) <\infty$ the estimators developed below in Theorem \ref{t stab N} are more computationally efficient than that of \cite{NK11,ZS01}. The main point  of Theorem \ref{t stab N} and conditions (N1)-(N3) is to certify analytically that consistent estimation is possible for a non-Gaussian noise and, in this sense, to support numerical experiments of \cite{NK11,ZS01}. 
\end{remar}

\section{Two-threshold estimators for parameters}
\label{s stab}

The estimators $\wh L_m$ of $\supp \mu$ constructed in the previous section consist of 
a finite number of \emph{maximal} superlevel $(h_m,J_m)$-arcs. 
We will say that a maximal $(h_m,J_m)$-arc is a 
\emph{localization arc}
if it contains at least one point of $\supp \mu$.

Assuming that an element $\om \in \Om$ belongs to 
the almost sure event $B_1$ of Proposition \ref{p det conver}, we see that for large enough
$m \ge M_0 = M_0 (\om)$ every localization arc contains exactly one point of $\supp \mu$. Hence,
the number 
of localization arcs is a strongly consistent estimator of the number 
$N$ of frequencies. Since localization arcs converge to one-point sets $\{ \ee^{\ii \la_n} \}$ 
as $m \to +\infty$, 
\emph{the middle point of the localization arcs are strongly consistent estimator of $z_n$.}

%
%

The problem arising in realization of this approach is possible presence 
of maximal $(h_m, J_m)$-arcs that do not contain any points of $\supp \mu$.
We will call them \emph{side arcs}. The convergence 
$\wh L_m \to \supp \mu$
implies that, with growth of $m$, side arcs have to lie in smaller and smaller neighborhoods 
of points $\ee^{\ii \la_n}$
and so in smaller and smaller neighborhoods of associated localization arcs.

The goal of this section is to provide a method that filters out side arcs, and so, detects localization 
arcs. This will lead to a construction of strongly consistent estimators of all parameters.

The filtering will require a special type of summation process and an additional level sequence 
$( h'_m )_{m=1}^{+\infty} $ of real numbers satisfying 
\begin{equation} \label{e hk pr}
h_m \le h'_m \qquad \text{ and } \qquad \lim (h'_m - h_m) = \lim \frac{H_m}{h'_m} = +\infty.
\end{equation}
 
\subsection{Almost decreasing anti-derivative kernels}
\label{ss:decrKer}

There can be two causes of appearance of side arcs: 
(i) random fluctuations of $S^A_{x,m} (\ee^{\ii \th})$ due to the noise component $S^A_{\ep,m}$,
and (ii) spectral leakage arising because of sidelobes (side peaks) of the anti-derivative kernel 
$|S^A_{\one, m} (\ee^{\ii \th})|$.

To eliminate the second effect, we will use kernels with a special property, which provide  bounds uniform in 
$m$ on fluctuations from the main lobe. Namely, additionally to assumptions 
 (\ref{e a 1})-(\ref{e a 4}), we will suppose that 
the kernel $S^A_{\one,n} (\ee^{\ii \th})$ is almost decreasing in the following sense:
for each $m \in \NN$,
\begin{gather} \label{e S=f+b}
S^A_{\one,m} (\ee^{\ii \th})  = f_m (\th) + \bb (m,\th) , \\
\text{ where }
f_m (\th) \text{ is a $2\pi$-periodic even  nonnegative function non-increasing on } [0,\pi],
\label{e fk}
 \\
\text{$\bb (m,\cdot) \in \L^\infty (\RR )$ and satisfy } | \bb |_{\infty} :=
\sup_{m \in \NN}   \| \bb (m,\th) \|_\infty < \infty .
 \label{e bk}
\end{gather}


An example of such a summation matrix is given by the following proposition.

\begin{prop} \label{p AP}
For a fixed constant $C>0$, let us consider a sequence $r_m = \ee^{-C/m}$, $m \in \NN$, and the truncated Poisson summation matrix $A$ defined by
\begin{equation} \label{e AP}
a_{m,t} = r_m^t \quad \text{ if } \  t \le m, \qquad a_{m,t} = 0 \quad \text{ if } \ t > m .
\end{equation}
Then
\begin{equation} \label{e S-ln}
|S^A_{\one,m} (\ee^{\ii \th}) + \ln (1-r_m \ee^{\ii \th}) | < \int_{C}^{+\infty}
\frac{\ee^{-x}}{x} \dd x , \quad m \in \NN,
\end{equation}
and the matrix $A$ satisfies (\ref{e S=f+b})-(\ref{e bk}).
\end{prop}

\begin{proof}
To prove (\ref{e S-ln}) it is enough to notice that
\[
|S^A_{\one,m} (\ee^{\ii \th}) + \ln (1-r_m \ee^{\ii \th}) | \le \sum_{t=m+1}^{+\infty} \frac{r_m^t}{t} < \int_{m}^{+\infty} 
\frac{\ee^{-Cx/m}}{x} \dd x .
\]
Then properties  (\ref{e S=f+b})-(\ref{e bk})  easily follow from (\ref{e S-ln}).
\end{proof}

It is obvious that $A$ defined by (\ref{e AP}) satisfies (\ref{e a 1})-(\ref{e a 4}).

\subsection{Strongly consistent estimators of $N$, $\ee^{\ii \la_n}$, and $\alpha_n$}
\label{ss:estNe}

In this subsection we assume that 
the matrix $A$ satisfies (\ref{e a 1})-(\ref{e a 3}) 
and the almost decreasing property (\ref{e S=f+b})-(\ref{e bk}).
We assume also that the sequences of $H_m$, $h_m$, $h'_m$, and $J_m$
are chosen in accordance with (\ref{e Hk1}), (\ref{e hk}), (\ref{e hk pr}), and (\ref{e:Jm}).


We will say that an $(h,J)$-arc $\Ar$ for $S^A_{u,m}$ is \emph{maximal} if it is not a subset of any other
$(h,J)$-arc for $S^A_{u,m}$. This means that either $\Ar$ contains all points of the form $\exp(\ii \frac{2\pi j}{J})$, $j \in \ZZ$, or $\Ar = \Arc \left[ \frac{2\pi j_1}{J} , \frac{2 \pi j_2}{J} \right]$ 
with $j_{1,2} \in \ZZ$ such that 
\begin{multline*}
\left| S_{u,m}^A \left( \exp \left(\ii \frac{2 \pi  (j_1-1) }{J} \right)\right)  \right| < h, \qquad 
\left| S_{u,m}^A \left( \exp \left(\ii \frac{2 \pi  (j_2+1) }{J} \right)\right)  \right| < h, \\
\text { and }
\left| S_{u,m}^A \left( \exp \left(\ii \frac{2 \pi  j}{J} \right)\right)  \right|
\ge h \text{ for all } j \in \ZZ \cap [j_1,j_2] .
\end{multline*}

\begin{defin} \label{d max cl}
By $\FAr_m (u)$ we denote the \emph{family of all maximal $(h_m,J_m)$-arcs for
$S^A_{u,m}$ that contain at least one $(h'_m,J_m)$-arc} for $S^A_{u,m}$. 
By $\wh N_m (u)$ we denote the
\emph{number of such maximal arcs}
in the family $\FAr_m (u)$. If $\FAr_m (u)$ is empty,  $\wh N_m (u)  := 0$.
\end{defin}

For the case when $u$ is the random signal $x = (x(t,\om))_{t=1}^{+\infty}$,
we will use the shortenings $\FAr_m = \FAr_m (x)$ and  $\wh N_m = \wh N_m (x) $.
The following theorem states, in particular, that for large  $m$, the random family  $\FAr_m$ 
coincides almost surely with the family of localization arcs, and that 
$\wh N_m$ is a strongly consistent estimator of $N$.

\begin{thm} \label{t stab N}
Assume  that at least one of the noise conditions (N1)-(N4) of 
Sections \ref{s:intro} and \ref{ss conver set} is fulfilled, and consequently, by Lemma \ref{l:NtoN0},
the $\om$-set $B_1$ defined in condition (N0) is an almost sure event. 

Then for $\om \in B_1$ the following statements hold:

\item[(i)] There exists $M = M (\om) \in \NN$ such that for all
$ m \ge M $,
\[
\text{$\wh N_m = N$  }
\]
and, for each $n \in \NN$ so that $1 \le n \le N $, 
\[
\text{ there exists a maximal $(h_m,J_m)$-arc 
$\Ar_{m,n} \in \FAr_m$ satisfying } \ \ee^{\ii \la_n} \in \Ar_{m,n} .
\]
This means that for $ m \ge M $, the family $\FAr_m$ consists exactly of 
$N$ maximal superlevel arcs  $\Ar_{m,1}$, \dots, 
$\Ar_{m,N}$. In the case $N=0$,  the family $\FAr_m$ is empty.

\item[(ii)] Suppose $N \neq 0$ and $m \ge M$. For each $n=1,\dots,N$,
denote by $j_{m,n}$ and $j'_{m,n}$ the natural numbers such that 
$\Ar_{m,n} = \Arc \left[ 2 \pi \ii \dfrac{j_{m,n}}{J_m} , 2 \pi \ii \dfrac{j'_{m,n}}{J_m} \right]$.
Then 
\begin{equation}
\wh z_{m,n} := \exp \left(2 \pi \ii \frac{j_{m,n}+j'_{m,n}}{2 J_m} \right) 
\quad \text{converge to $z_n =\ee^{\ii \la_n}$
\quad as $m \to +\infty$}. \label{e:whzn}
\end{equation}
If, additionally to (\ref{e:Jm}), the sequence $( J_m )_{m=1}^{+\infty}$ satisfies 
\begin{equation} \label{e j comp max}
\lim_{m\to +\infty} \frac{S^A_{\one,m} \left( \ee^{2 \pi \ii / J_m} \right)}{S^A_{\one,m} (1)} = 1 ,
\end{equation}
and 
$j_{m,n}^{\max} \in [j_{m,n},j'_{m,n}] \cap \ZZ$ is a maximizer in the sense   
\begin{equation} \label{e:maximizer}
\left| S^A_{x,m} \left( \exp \left(  2 \pi \ii \frac{j_{m,n}^{\max}}{J_m} \right) \right) \right| =
\max_{j_{m,n} \le j \le j'_{m,n}} \left| 
S_{x,m}^A  \left( \exp \left(  2 \pi \ii \frac{j}{J_m} \right) \right) 
\right| ,
\end{equation}
then 
\begin{equation} \label{e al m}
\al_n = \lim_{m \to +\infty} 
\frac{S^A_{x,m} \left( \exp \left(  2 \pi \ii \dfrac{j_{m,n}^{\max}}{J_m} \right) \right)}
{\sum_{t=1}^{m} a_{m,t} / t} .
\end{equation}
\end{thm}

The proof is given in the next subsection. The existence of 
$(h'_m)_{m=1}^{+\infty}$ and $(J_m)_{m=1}^{+\infty}$ satisfying 
(\ref{e hk pr}) and (\ref{e j comp max}) is obvious. A particular example is given by the next statement.

\begin{prop} \label{p seq AP}
Let $A$ be the truncated Poisson summation matrix defined by
\begin{equation} \label{e AP C=1}
a_{m,t} =\ee^{-t/m} \text{ if } t \le m, \qquad a_{m,t} = 0 \text{ if } t > m .
\end{equation}
Then conditions (\ref{e Hk1}), (\ref{e hk}), (\ref{e:Jm}), and (\ref{e hk pr})
are fulfilled if $H_m$, $h_m$, $h'_m$ are defined  by
\begin{gather} \label{e:APparam}
H_m = \frac{1}{2} \ln (m) - \frac{1}{2} , \quad
h_m = c_3 \ln^{1-\de} (m) + c_4, \quad  h'_m =  c_5 \ln^{1-\de} (m) + c_6, 
\end{gather}
where $\de \in (0,1)$ and $c_{3,4,5,6} \in \RR$ are arbitrary constants such that
$0<c_3 < c_5$, $c_4 < c_6$,
and $J_m$ satisfy 
$J_m \ge  2 \pi  \sqrt{m}  $.

If, additionally,
\begin{equation} \label{e:Jm_large}
J_m \ge    2 \pi m^{3/2}  ,
\end{equation}
then (\ref{e j comp max}) is also fulfilled.
\end{prop}

The proof is given in Section \ref{ss:ProofsPrseqAP}.

\begin{remar} \label{r max clust}
It is clear that the following statements involving a maximal $(h_m,J_m)$-arc $\Ar$ 
for $S^A_{u,m}$ are equivalent:
(i) $\Ar$ contains at least one $(h'_m,J_m)$-arc,
(ii) $\Ar$ contains at least one maximal $(h'_m,J_m)$-arc, (iii) $\Ar$ contains
a point $z_0$ of form $z_0 = \exp \left(2 \pi \ii  \frac{ j}{J_m }\right)$
with $j \in \ZZ$ such that $|S^A_{u,m} (z_0)| \ge h'_m$.
\end{remar}

\begin{remar} \label{r L'k}
For large $m$, the sets
$\wh L'_m $ defined as the unions of all $(h'_m, J_m)$-arcs for $S_{x,m}^A$
obviously provide a better localization for $\supp \mu$ than $\wh L_m$.
Namely, for $\om \in B_1$,
\begin{eqnarray*}
 \ee^{\ii \la_n} \in \wh L'_m \cap \Ar_{m,n} \text{ for large enough } m, 
 \quad \text{ and } \quad \{ \ee^{\ii \la_n} \}  =  \lim_{m\to +\infty} \wh L'_m \cap \Ar_{m,n} .
\end{eqnarray*}
Since in the settings of (\ref{e:maximizer}), 
$\exp \left(  2 \pi \ii \frac{j_{m,n}^{\max}}{J_m} \right) \in \Ar_{m,n} \cap \wh L'_m$, 
the procedure of finding of \\
$\max\limits_{j_{m,n} \le j \le j'_{m,n}} \left| 
S_{x,m}^A  \left( \exp \left(  2 \pi \ii \frac{j}{J_m} \right) \right) 
\right|$ can be restricted to the numbers 
$\exp \left(  2 \pi \ii \frac{j}{J_m} \right)$ lying in the more narrow set $ \Ar_{m,n} \cap \wh L'_m$.
\end{remar}

 Assume that a finite signal $\{x(t)\}_{t=1}^M$ of a  sufficiently large length $M$ is given and, for $m \le M$, 
the summation matrix $(a_{m,t})$, the levels $h_m, h'_m$, and the frequency grid orders $J_m$ are prescribed in accordance with (\ref{e AP C=1})-(\ref{e:Jm_large}).
If $m$ is large enough to ensure that at least one of the numbers $ \left| 
S_{x,m}^A  \left( \exp \left(  2 \pi \ii \frac{j}{J_m} \right) \right) \right| $, $j=1,\dots,J_m$,
is less than the level $h_m$, then the algorithmic summary for computation of the estimators of $z_n$ and $\al_n$, $n=1,\dots,N$, by first $m$ values of $x$ is as follows.
 

\begin{algorithm} \quad

\medskip

 \noindent {\bf{Input:}}  $m, J_m \in \NN$, $x(t) \in \CC$ and  $a_{m,t} \in (0,1]$ ($t=1,\dots, m$),  $h_m, h'_m  \in \RR_+$.\\[1ex]
\noindent \textbf{Step 1.} Compute  $ S_{x,m}^A  \left( \ee^{ 2 \pi \ii  j/J_m} \right) $ for $j=1$, \dots, $J_m$ by formula (\ref{e SAxm}).\\[1ex]
\noindent \textbf{Step 2.} Find all $j \in [1,J_m] \cap \NN $ such that $\left| S_{x,m}^A  \left( \ee^{ 2 \pi \ii  j/J_m} \right) \right| \ge h_m$.\\[1ex]
\noindent \textbf{Step 3.} Find all  maximal $(h_m,J_m)$-arcs using Definition \ref{d cluster}.\\[1ex]
 \noindent \textbf{Step 4.} Find the family $\FAr_m (x) = \{ \Ar_{m,n} \}_{n=1}^{\wh N_m} $ ($\wh N_m \in \ZZ^+$) of all maximal  $(h_m,J_m)$-arcs
 $\Ar_{m,n} = \Arc \left[ 2 \pi \ii \dfrac{j_{m,n}}{J_m} , 2 \pi \ii \dfrac{j'_{m,n}}{J_m} \right]$ ($j_{m,n}, j'_{m,n} \in \NN$) that, according to 
 Definition \ref{d max cl},  contain
  a point $\zeta_n$ of form $\zeta_n = \exp \left(2 \pi \ii  \frac{j}{J_m }\right)$, $j \in [j_{m,n},j'_{m,n}] \cap \NN$,
   satisfying $|S^A_{u,m} (\zeta_n)| \ge h'_m$.\\[1ex]
  \noindent {\bf{Output:}} the estimator $\wh N_m$ of the number of frequencies.\\[1ex]
 \noindent \textbf{Step 5.} If $\wh N_m =0$, then \emph{stop}.\\ Otherwise (when $\wh N_m \ge 1$) continue with the following steps:\\[1ex]
 \noindent \textbf{Step 6.} For   $n=1$, \dots, $\wh N_m$, compute the estimator
$\wh z_{m,n} := \exp \left(2 \pi \ii \frac{j_{m,n}+j'_{m,n}}{2 J_m} \right)$.\\[1ex]
\noindent \textbf{Step 7.} For   $n=1$, \dots, $\wh N_m$, compute the estimator
$\wh \alpha_{m,n} $  by  the averaging formula\\[0ex]
$\quad \quad \wh \alpha_{m,n} = \frac{1}{I \sum_{t=1}^{m} a_{m,t} / t} \sum_{i=1}^{I} S^A_{x,m} \left( \exp \left(  2 \pi \ii \dfrac{j_{m,n}^{\max,i}}{J_m} \right) \right) 
$, where $\{ j_{m,n}^{\max,i} \}_{i=1}^{I} $  is the set \linebreak $\argmax\limits_{j \in [j_{m,n}, j'_{m,n}] \cap \ZZ} \left| S_{x,m}^A  \left( \exp \left(  2 \pi \ii \frac{j}{J_m} \right) \right) 
\right|$ of maximizers in the sense  of (\ref{e:maximizer}), and $I \in \NN$ is their number.\\[1ex]
\noindent {\bf{Output:} } $\wh z_{m,n}, \wh \alpha_{m,n}$, $n = 1,\dots,\wh N$.
\end{algorithm}
\begin{remar}
Actually, by the proof of Theorem \ref{t stab N}, any rule choosing a point $\wh z_{m,n} \in \Ar_{m,n}$ provides consistent estimators $\wh z_{m,n}$. The formula $\wh z_{m,n} = \exp \left(2 \pi \ii \frac{j_{m,n}+j'_{m,n}}{2 J_m} \right)$ was chosen only because of its simplicity. A more practical choice may involve the number
$j_{m,n}^{\max} \in [j_{m,n},j'_{m,n}] \cap \NN$ that is a maximizer in the sense  of  
(\ref{e:maximizer}), or the average of such maximizers in a very improbable case of several of them.
\end{remar}

\subsection{Proof of Theorem \ref{t stab N}}
\label{ss:ProofsThPr}

When $N=0$, Theorem \ref{t stab N} follows from  Proposition \ref{p det conver} (ii).
Consider the case $N > 0$. Let us fix natural $n \le N$ and define the sequence $(v(t))_{t=1}^{+\infty}$ 
and the measure $\mu_n$ by (\ref{e:u_mu_n}).

Statement (i) of Theorem \ref{t stab N} follows from Proposition \ref{p det conver}
and the next proposition.

\begin{prop}  \label{p isol max cl}
Let $\Arc [\la_n - \ga, \la_n+\ga]$ with a certain $\ga>0$ be disjoint with
the support of the measure $\mu_n$, and let $\om \in B_1$. Then: 

\item[(a)] There exists  a natural number $M_n = M_n (\ga,\om)$ such that
for $m \ge M_n $ the following statements hold:
\subitem (i) in the family $\FAr_m$ there exists exactly one maximal
$(h_m,J_m)$-arc $\Ar_{m,n}$ that has a nonempty
intersection with $\Arc [\la_n - \ga, \la_n+\ga]$,
\subitem (ii)  $\Ar_{m,n} \subset \Arc [\la_n - \ga, \la_n+\ga]$,
\subitem (iii)  $\ee^{\ii \la_n} \in \Ar_{m,n}$.

\item[(b)]  As $m \to +\infty$, $\Ar_{m,n}$ converge to the one-point set 
$\{ \ee^{\ii \la_n} \}$ .
\end{prop}

\begin{proof}
Proposition \ref{p det conver} and Remark \ref{r max clust} applied to
$(h_m,J_m)$-arcs and to $(h'_m,J_m)$-arcs imply the existence of a number $m_3 $ 
such that for all $m \ge m_3$ there exist a maximal $(h'_m,J_m)$-arc $\Ar'_{m,n}$ and a
maximal $(h_m,J_m)$-arc $\Ar_{m,n}$ such that
\begin{equation} \label{e lam in Cl subs Cl}
\ee^{\ii \la_n} \in \Ar'_{m,n} \subset \Ar_{m,n} \subset \Arc \left[ \la_n - \ga +\frac{2 \pi}{J_m}, \la_n + \ga -\frac{2 \pi}{J_m} \right] .
\end{equation}
These inclusions imply $\Ar_{m,n} \in \FAr_m$ and the statements (a.ii)-(a.iii).

To prove (a.i), assume that $m \ge m_3$,
and that there exists 
\begin{equation} \label{e:th0}
\ee^{\ii \th_0} \in \Arc [ \la_n - \ga , \la_n + \ga ] \setminus \Ar_{m,n}
\text{ so that $\th_0$ is of the form } \frac{2 \pi j}{J_m} ,  \quad j \in \ZZ .
\end{equation}
Then
\begin{equation} \label{e th0 x-u u}
|S^A_{x,m} (\ee^{\ii \th_0}) | \le |\al_n S^A_{\one, m} (\ee^{\ii (\th_0 - \la_n)})|
+ |S^A_{x-v, m} (\ee^{\ii \th_0 })| \le |\al_n| \left( f_{m} (\th_0 -\la_n) + | \bb |_{\infty} \right) 
+ C_1 ,
\end{equation}
where $C_1$ is the finite number defined by (\ref{e C1 sup max}).

Since $\ee^{\ii \la_n} \in \Ar_{m,n}$, $\ee^{\ii \th_0} \not \in \Ar_{m,n}$, and
$ \Ar_{m,n}$ is a \emph{maximal} $(h_m,J_m)$-arc, one can see that
the subarc of $\Arc [ \la_n - \ga , \la_n + \ga ]$ connecting $\ee^{\ii \la_n}$ and 
$\ee^{\ii \th_0}$ contains a point $\ee^{\ii \th_1}$ such that
\begin{equation}
|S^A_{x,m} (\ee^{\ii \th_1}) | < h_m . \label{e th1<hm}
\end{equation}
Since $f_{m}$ is monotone on $[-\pi,0] $ and $[0,\pi]$, we see that
\begin{equation}
f_{m} (\th_0 - \la_n) \le f_{m} (\th_1 - \la_n) \label{e fnm th0th1}
\end{equation}
where, if necessary,  a multiple of $2\pi$ is added to $\th_{0,1}$ to ensure 
$|\th_i - \la_n| \le \pi$, $i=0,1$.
On the other side,
\[
 f_{m} (\th_1 - \la_n)  \le |S^A_{\one, m} (\ee^{\ii (\th_1 - \la_n)})| + | \bb |_{\infty} =
|\al_n|^{-1} |S^A_{v, m} (\ee^{\ii \th_1 })| + | \bb |_{\infty}
\]
and
\[
|S^A_{v, m} (\ee^{\ii \th_1 })| \le  |S^A_{x, m} (\ee^{\ii \th_1 })| +C_1 .
\]
Combining these inequalities with (\ref{e th1<hm}) and (\ref{e fnm th0th1}), we see that
\[
f_m (\th_0 - \la_n)  \le |\al_n|^{-1} (h_m + C_1) + | \bb |_{\infty}.
\]
The latter and (\ref{e th0 x-u u}) imply
\[
|S^A_{x,m} (\ee^{\ii \th_0}) | \le h_m + 2 C_1 + 2 | \bb |_{\infty} \; |\al_n|   .
\]
By (\ref{e hk pr}), there exists $m_4 \ge m_3$ such that
$h'_m -h_m > 2 C_1 + 2 | \bb |_{\infty} \; |\al_n|$
for $m \ge m_4$. For such $m$ we have $ |S^A_{x,m} (\ee^{\ii \th_0}) | < h'_m$.
Since the last inequality is valid for any $\ee^{\ii \th_0} $ satisfying 
(\ref{e:th0}), we see that for $m \ge m_4$ there are no $(h'_m,J_m)$-arcs
lying outside $\Ar_{m,n}$ and intersecting $\Arc [ \la_n - \ga , \la_n + \ga ]$.
Thus, (a.i) is proved.

The statement (b) follows from (a.i) and (\ref{e lam in Cl subs Cl}). This completes the proof.
\end{proof}

Let us prove statement (ii) of Theorem \ref{t stab N}.
The convergence of (\ref{e:whzn}) is obvious from Proposition \ref{p isol max cl}.
Now, we assume (\ref{e j comp max}) and prove (\ref{e al m}).

Let us denote $\th_{m,j} :=2 \pi \frac{j}{J_m} $ for $j \in [j_{m,n},j'_{m,n}] \cap \ZZ$ and 
$\th_m^{\max} := 2 \pi \dfrac{j_{m,n}^{\max}}{J_m}$.
Note that
\begin{equation} \label{e sum a/k = S}
\sum_{t=1}^m a_{m,t} /t = S^A_{\one,m} (1) = \| S^A_{\one,m} (\ee^{\ii \th}) \|_\infty =
\max_{\th \in \RR} \left| \sum_{t=1}^{m} \cos (t \th) a_{m,t} /t \right| .
\end{equation}
and that
\begin{eqnarray*}
\lim_{m \to +\infty} \frac{S^A_{x,m} (\ee^{\ii \th_m^{\max}})}{\sum_{t=1}^{m} a_{m,t} /t} & = & 
\lim_{m \to +\infty}  \frac{
\al_n f_{m} (\th_m^{\max}- \la_n) +
\al_n  \bb_{m} (\th_m^{\max}- \la_n) +S^A_{x-v,m} (\ee^{\ii \th_m^{\max}})}{S^A_{\one,m} (1)}
\\ & = &
\al_n \lim_{m \to +\infty} \frac{f_{m} (\th_m^{\max}- \la_n) }{S^A_{\one,m} (1)}
\end{eqnarray*}
if the latter limit exists.
Indeed, the sequences $ \bb_{m} (\th_m^{\max}- \la_n)$  and $S^A_{x-v,m} (\ee^{\ii \th_m^{\max}})$
are bounded due to (\ref{e bk}) and $\om \in B_1$, respectively.

Since $f_{m} (\th_m^{\max}- \la_n) \le S^A_{\one,m} (1) +| \bb |_{\infty}$, we see that
$\limsup\limits_{m\to\infty} \dfrac{f_{m} (\th_m^{\max}- \la_n) }{S^A_{\one,m} (1)} \le 1$.
Thus, to prove  (\ref{e al m}) it is enough to show that
\begin{equation} \label{e liminf>1}
\liminf_{m \to +\infty} \frac{f_{m} (\th_m^{\max}- \la_n) }{S^A_{\one,m} (1)} \ge 1.
\end{equation}
Since $C_1 := \sup\limits_{m\in \NN} \max\limits_{|\th-\la_n|<\ga} |S_{x-v,m} (\ee^{\ii \th})| < \infty$,
we see that for $\th \in [\la_n -\ga,\la_n+\ga]$,
\begin{equation} \label{e S - am f}
|S^A_{x,m} (\ee^{\ii \th}) - \al_n f_{m} (\th-\la_n)| \le C_1 +  | \bb |_{\infty} |\al_n|.
\end{equation}
Hence, for $\wt \th_m \in \{ \th_{m,j} \}_{j=j_{m,n}}^{j'_{m,n}}$ satisfying 
$| \ee^{\ii \wt \th_m} - \ee^{\ii \la_n} | = 
\min\limits_{j_{m,n} \le j \le j'_{m,n}} |\ee^{\ii \wt \th_{m,n}} - \ee^{\ii \la_n}|$,
one has
\[
|\al_n| \; f_m (\th_m^{\max} - \la_n ) \ge
|S^A_{x,m} (\ee^{\ii \th_m^{\max}})| -C_1 -  | \bb |_{\infty} |\al_n|\ge
|S^A_{x,m} (\ee^{\ii \wt \th_m})| -C_1 -  | \bb |_{\infty} |\al_n|,
\]
where the definition of $\th_m^{\max}$ was used to obtain the last inequality.
Applying (\ref{e S - am f}) once again, we see that
\begin{equation*} 
|\al_n| \; f_m (\th_m^{\max} - \la_n) \ge |\al_n | \; f_{m} ( \wt \th_m - \la_n ) -2 C_1 - 2 |\bb |_\infty |\al_n| .
\end{equation*}
Therefore,
\begin{equation} \label{e liminf>liminf}
\liminf_{m \to +\infty} \frac{f_{m} (\th_m^{\max}- \la_n) }{S^A_{\one,m} (1)} \ge
\liminf_{m \to +\infty} \frac{f_{m} (\wt \th_m- \la_n) }{S^A_{\one,m} (1)} .
\end{equation}
From the monotonicity of $f_{m}$ on $[-\pi,0]$ and $[0,\pi]$, one can see that
\[
f_m (2\pi/ J_m ) \le f_m (\wt \th_m - \la_n) \le f_m (0) \le
S^A_{\one,m} (1) + | \bb |_{\infty} .
\]
This together with (\ref{e j comp max}) and (\ref{e liminf>liminf}) implies  (\ref{e liminf>1}) 
and completes the proofs of statement (\ref{e al m}) and of the theorem.

\subsection{Proof of Proposition  \ref{p seq AP}}
\label{ss:ProofsPrseqAP}

The following lemma can be obtained from (\ref{e S-ln}) and elementary estimates.

\begin{lem} \label{l:APf}
For the matrix (\ref{e AP C=1}), the functions $f_m$ of (\ref{e S=f+b})-(\ref{e bk})
can be taken as
$f_m (\th) = \max\{ 0 ,  -\ln |1-\ee^{-1/m+\ii \th}|\} $. In this case, $| \bb |_\infty < \ln 2 + \pi/3 + \ee^{-1} < 11/5$.
Moreover,
\begin{equation} \label{e:ln-e}
|S^A_{\one,m} (\ee^{\ii \th})| \ge - \ln |1-\ee^{-1/m +\ii \th}|- \ee^{-1}
\quad  \text{ for } |\th|\le \pi/3 .
\end{equation}
\end{lem}

Now, the proof of Proposition  \ref{p seq AP} consists of three steps.

\emph{Step 1.} Let us show that
\begin{equation} \label{e Sth>=}
|S^A_{\one,m} (\ee^{\ii \th})| \ \ge \ \frac{p+1}{2} \ \ln m \ - \  \frac12 \ \quad \text{ for } \quad |\th| \le \frac{1}{m^{p+1/2}}, \
p =0,1, \ m \in \NN .
\end{equation}
From $|\th| \le 1 < \pi/3$ and (\ref{e:ln-e}), one can see that
it is enough to prove
$|1-\ee^{-1/m+\ii \th}|^2 \le 1/m^{p+1}$ or, equivalently,
$4 \sin^2 (\th/2) \le \ee^{1/m}/ m^{p+1} +2- (\ee^{1/m}+\ee^{-1/m})$. The latter follows for
$ |\th| \le 1/m^{p+1/2}$ from $|\sin (\th/2)| \le |\th|/2$ and the fact that
\[
x^{p+1} \ee^x +2 - (\ee^{x}+\ee^{-x}) \ge x^{2p+1} \text{ for  } x \ge 0  \text{ and } p=0,1.
\]

\emph{Step 2.} Let $H_m = \frac12 \ln m - \frac 12$.
Then (\ref{e Sth>=}) with $p=0$ implies (\ref{e:Jm}) for $J_m \ge 2 \pi \sqrt{m} $, and so also (\ref{e Hk1}).
Taking $h_m$ and $h'_m$ defined by (\ref{e:APparam}), one ensures (\ref{e hk pr}).

\emph{Step 3.} Let us put $\th_m = 2\pi/\wt J_m$ with $\wt J_m \in \NN$ such that
\begin{equation} \label{e:SJ>=lnk}
|S^A_{\one,m} (\ee^{\ii \th_m})| \ \ge \ \ln m \ - \  \frac12 .
\end{equation}
Applying (\ref{e Sth>=}) with $p=1$, we see that (\ref{e:SJ>=lnk}) holds whenever
$\wt J_m  \ge m^{3/2}$. 
This means that for $J_m \ge 2 \pi m^{3/2}$,
we obtain with the use of (\ref{e S-ln}) that
\begin{equation} \label{e:|Sthk|/S1=1}
1 \ge \frac{|S^A_{\one,m} (\ee^{\ii 2\pi/ J_m})|}{S^A_{\one,m} (1)} \ge
\frac{\ln m -1/2}{- \ln (1-\ee^{-1/m}) +1/\ee}
\text{ and, in turn, } \lim_{m\to\infty}\frac{|S^A_{\one,m} (\ee^{\ii 2\pi/J_m})|}{S^A_{\one,m} (1)} = 1 .
\end{equation}
It follows from $J_m \to +\infty$ and (\ref{e S-ln}) that
$\lim_{m\to\infty}\frac{S^A_{\one,m} (\ee^{\ii 2\pi/J_m})}{|S^A_{\one,m} (\ee^{\ii 2\pi/J_m}) |} = 1$.
This and (\ref{e:|Sthk|/S1=1}) complete the proof of (\ref{e j comp max}) and of the proposition.

\section{Discussion and additional remarks}
\label{s:Discussion}

\subsection{Relaxation of signal and noise assumptions}

While the case of decaying useful part $(y(t))_{t=1}^\infty$ of the signal was addressed in Statistics in connection with  significance of related statistical hypotheses \cite{B87}, we are not aware of  studies of consistency of estimators for such models. Models with decaying $(y(t))_{t=1}^\infty$ appears, e.g.,  in the analysis of chirp-type signals \cite{MT03,Aetal16} (after inversion of the time direction) or 
in the passing from idealized conservative models 
of various tomography techniques to models with attenuated waves 
\cite{HK15}. 

Signal processing intuition suggests that if the decaying deterministic signal (i.e., $y(t) \to 0$ as $t\to \infty$) is corrupted by a random noise $\ep_t$ of `constant strength`, then `the signal-to-noise ratio' goes to $0$ as $t \to \infty$ and it is difficult to expect that standard estimators for such a signal remain asymptotically consistent.

On the other side, an estimator as a function of $\{x(t)\}_{t=1}^m$ may cancel to some extend the influence of the random noise sequence $\{ \ep_t \}_{t=1}^m$. Theorem \ref{t stab N}
shows that this cancellation can be strong enough to recover parameters of $y(\cdot)$ if the decay of $y(\cdot)$ is not very fast.
Indeed, consider the model $x(t) = \sum_{n=1}^N \al_n \frac{\ee^{-\ii \la_n t}}{t^\xi} +\ep _t$, where $\la_n$ and $\al_n$ are as before, $\ep_t $ are i.i.d. symmetric  with finite variances, and  $\xi \in \RR_+ $ is the rate of power decay of the deterministic terms $y(t) = \sum_{n=1}^N \al_n \frac{\ee^{-\ii \la_n t}}{t^\xi}$. Then the rescaled signal 
$\wt x (t) = t^\xi x(t)$ is of form (\ref{e x}).  If $\xi < 1/2$, the rescaled noise sequence $\wt \ep_t := t^{\xi} \ep_t$ satisfies (N3), and so, Theorem \ref{t stab N} is applicable.

This observation and the sequence of assumptions (N1)-(N3) naturally lead to the following question.
\begin{prob}
To what extend is it possible to relax the assumptions on the random variables $\ep_t$ and the model for deterministic part $y(\cdot)$ of the signal
such that the recovery of the parameters of $y(\cdot)$ by a certain asymptotically consistent procedure remains possible.
\end{prob} 

Another way to relax assumptions on the model for $y(\cdot)$ is to allow the number $N$ of frequencies to be (countably) infinite.
The assumption that $N$ is finite may be non-reasonable for some of applications.
For example, many of models of mechanical, acoustical,
and electro-magnetic resonators lead to infinite number of eigen-frequencies \cite{HK15,KHC13}. Generic excitation
generates  a signal containing all of them. While only some finite number of frequencies
have large enough amplitudes to be interesting, the infinite sum of the rest still contribute to
the signal. This contribution cannot be treated as independent random noise terms $\ep_t$.
That is why it is reasonable to include signals with possibly infinite $N$ into consideration,
but to recover  only frequencies with comparatively large amplitudes $|\al_n|$. 

It is not difficult to see that the consistent localization statement 
($\lim\limits_{m\to \infty} \wh L_m = \supp \mu$ a.s.)
of Theorem \ref{t conver set} remains valid in the case when $N = \infty$ under the additional assumptions that 
$ \sum_{n=1}^\infty |\al_n| < \infty $ and every frequency $z_n$ is an isolated point of the 
set $\{ z_n \}_{n=1}^\infty$ of all frequencies.

\subsection{Probability estimates for accurate localization}
\label{s:prob_ est_acloc}

The study of the precision of the above estimators
 requires knowledge of probability bounds on distribution tails of $\L^\infty (\TT)$-norms of random trigonometric polynomials $S^A_{\ep,m} (\ee^{\ii \th})$ constructed by the noise sequence $(\ep_t)_{t=1}^m$.
To illustrate this, consider probability estimates on accuracy of the procedure of localization of $\supp \mu = \{ z_n \}_{n=1}^N$ by the union $\wh L_m$ of superlevel arcs (see  Section \ref{ss:loc}).
 
In the sequel, we assume the truncated power summation matrix is defined by (\ref{e AP C=1}), 
and that $H_m$, $h_m$, $J_m$, and $h'_m$ are chosen in accordance with (\ref{e Hk1})-(\ref{e:Jm}) and (\ref{e hk pr}). 
For $\th_{1,2} \in \RR$, let us denote by $\dist_\TT (\ee^{\ii \th_1}, \ee^{\ii \th_2}) $ (by $\dist_\TT (\ee^{\ii \th_1},G) $) the `circular' distances between
$\ee^{\ii \th_1}$ and $\ee^{\ii \th_2}$ (resp., $\ee^{\ii \th_1}$ and the set $G \subset \TT $).
More precisely, $
\dist_\TT (\ee^{\ii \th_1}, \ee^{\ii \th_2}) := \min_{k \in \NN} |\th_2- \th_1 - 2 \pi k| $,
$ \dist_\TT (\ee^{\ii \th_1}, G) := \inf_{z \in G} \dist_\TT (\ee^{\ii \th_1},z)  $.
Recall that $\|\mu\| = \sum_{n=1}^\infty |\al_n| = \int_\TT |\dd \mu| $.

\begin{prop} \label{p:P_loc_gen}
(i) Assume that $\ee^{\ii \la_n}$ is the only point of $\supp \mu$ in $\Arc (\la_n -\De,\la_n+\De)$
and that $\De > 2 \pi/ J_m$. Then
\begin{equation} \label{e P lam notin}
\P (\ee^{\ii \la_n} \not \in \wh L_m) \le  \P \Bigl(
\| S^A_{\ep,m}\|_\infty \ge |\al_n| H_m - h_m - (\|\mu \| - |\al_n|) \, \left(f_{m} (\De - 2 \pi/ J_m)
+| \bb |_{\infty} \right) \,
\Bigr) .
\end{equation}

(ii) Assume that $z\in \TT$ is such that
$\dist_\TT (z,\supp \mu ) \ge \De \ge 2\pi/ J_m$.  Then
\begin{equation} \label{e P no mu}
\P (z \in L_m ) \le
\P \Bigl(
\| S^A_{\ep,m}\|_\infty \ge h_m - \|\mu \| \, \left( f_{m} (\De -\pi/J_m) +| \bb |_{\infty}
\right) \,
\Bigr)  .
\end{equation}

Statements (i) and (ii) also hold if one changes $L_m$ and $h_m$ to $L'_m$ and $h'_m$, respectively.
\end{prop}

\begin{proof}
\textbf{(i)}  Put $v(t) = \al_n \ee^{-\ii  \la_n t}$. Since $\De > 2\pi/J_m$, there exist
$\th_i = 2\pi l_i / J_m$, $i =1,2$, such that $l_i \in \ZZ$,  $l_2 = l_1 +1$,
and $\ee^{\ii \la_n} \in \Arc [\th_1, \th_2]$.

Suppose $\ee^{\ii \la_n} \not \in \wh L_m$. Then at least one of the two inequalities $|S^A_{x,m}(\ee^{\ii \th_i})| < h_m$, $i=1,2$, holds. Let us fix such $\th_i$ and note that the corresponding
inequality implies
\begin{multline*} 
|S^A_{\ep,m}(\ee^{\ii \th_i})| \ge |S^A_{v,m}(\ee^{\ii \th_i})| - h_m-|S^A_{y-v,m}(\ee^{\ii \th_i})|
\ge \\ \ge |\al_n| H_m - h_m - (\|\mu \| - \al_n) (f_{m} (\De - 2\pi/J_m) + | \bb |_{\infty}) .
\end{multline*}
To show the last estimate, it is enough to notice that
$\{ \ee^{\ii \la_j }\}_{j \neq n} \cap \Arc [\th_i - \De + 2\pi/J_m, \th_i + \De - 2\pi/J_m] = \varnothing$
and that (\ref{e S=f+b})-(\ref{e bk}) imply
\[
|S^A_{y-v,m}(\ee^{\ii \th_i})| \le \sum_{j \neq n} |\al_j| (f_{m} (\th_i - \la_j) + | \bb |_{\infty}) \le
\sum_{j \neq n} (f_{m} (\De - 2\pi/J_m) +| \bb |_{\infty}) \sum_{j \neq n}  |\al_j|  .
\]

\textbf{(ii)}
Assume that $z \in \Arc [\th_1, \th_2] $, where $\th_{1,2} =2\pi l_{1,2} /J_m$ with
$l_{1,2} \in \ZZ$ and $l_2 = l_1 + 1$.
Then $z \in \wh L_m$ if and only if $|S^A_{x,m} (\ee^{\ii \th_i}) | \ge h_m$ for $i=1,2$.
Since $\dist_\TT (z,\ee^{\ii \th_i}) \le \pi /J_m$ for at least one of $i={1,2}$,
we see that, for this $i$,
$|S^A_{\ep,m}(\ee^{\ii \th_i})| \ge  h_m - |S^A_{y,m}(\ee^{\ii \th_i})|
\ge h_m - \| \mu \| \, (f_{m} (\De- \pi/J_m) + | \bb |_{\infty}).$
\end{proof}

The estimates on the distribution tails of random polynomials rely on concentration inequalities and, up to our knowledge, are available only under the subgaussian or subexponential assumptions on $\ep_t$. 

Let us consider the subgaussian case. Recall that a real random variable $\xi$ is called \emph{subgaussian with the scaling factor $b \ge 0$} if
$\E (e^{t\xi}) \le \ee^{b^2 t^2/2}$ for all $t \in \RR$.

\begin{lem} \label{l P |q|>C}
Consider a complex trigonometric polynomial
$
q(\th) = \sum_{k=0}^n a_k \xi_{k} \ee^{\ii k \th} ,
$
where $a_k \in \CC$ and where $\xi_k$ are i.i.d. complex random variables such that $\re \xi_k$ and $\im \xi_k$ are independent and subgaussian with the scaling factors $b_1$ and $b_2$, respectively.
Assume that $b_{1,2}$ are chosen such that $b_1\ge b_2 \ge 0$ and $b_1>0$.
Put \linebreak[4]
$r:= b_1^2 \sum_{k=0}^n |a_k|^2 $. 
Then,
$\P (\| q \|_\infty \ge C) \le  (n+1) 2 \pi \ee^{3/2} \left( \frac{C^2}{2r} -1 \right) \,   \exp \left( - \frac{C^2}{4 r} \right)$ 
whenever 
$C>2\sqrt{2r}$.
\end{lem}

This lemma can be obtained by modification of the subgaussian arguments of Kahane \cite[Section 6.2]{K85} (which are based on the Salem-Zygmund use of Bernstein inequalities \cite{SZ54}).

From now on, suppose that the noise satisfies the following assumption:
\begin{itemize}
\item[(Nsg)] $\ep_t$ are i.i.d. complex random variables such that for a certain
$\vphi \in [-\pi,\pi)$ the random variables
$\re (\ee^{\ii \vphi }\ep_k)$ and $\im (\ee^{\ii \vphi }\ep_k)$ are
independent and subgaussian with the scaling factors $b_1$ and $b_2$, respectively,
chosen such that $b_1\ge b_2 \ge 0$ and $b_1>0$.
\end{itemize}

Lemma \ref{l P |q|>C} implies that under this assumption
\begin{gather} \label{e P|q|>C_AP1}
\P (\| S^A_{\ep,m} (e^{\ii \th}) \|_\infty \ge C) 
 \le
 \frac{\pi \ee^{3/2}}{b_1^2 \sum_{t=1}^m a_{m,t}^2/t^2} \, m C^2 \exp \left( - \frac{3}{2 \pi^2 b_1^2} C^2 \right)
\end{gather}
for $C \ge 2 \pi b_1 / \sqrt{3} $ (we have used here the facts that $a_{m,t} \le 1$ and $\sum_{t=1}^m a_{m,t}^2/t^2 < \pi^2 / 6$).

Let $O_\De (\supp \mu)$ be the open in the topology of $\TT$ neighborhood of $\supp \mu$ defined by
\[
O_\De (\supp \mu) := \{ z \in \TT \ : \ \dist_\TT (z,\supp \mu) < \De\} , \qquad \De >0 .
\]
By $\A (\wt \alpha,\De)$ we denote the set of all points $z$ of $\supp \mu$ that bear
a complex mass $\alpha$ with $|\alpha| > \wt \alpha$ and satisfy $\dist_\TT (z,\supp \mu \setminus \{ z \}) \ge \De$.

The above bound (\ref{e P|q|>C_AP1}) on the distribution tail of $\| S^A_{\ep,m} (\ee^{\ii \th}) \|_\infty$ can be combined with Proposition \ref{p:P_loc_gen} to obtain, for instance, the following result.

\begin{cor} \label{c:PA_Pdist}
Let $m \ge 3$, $h_m= \ln^{1-\de} m $ with
$\de \in (0,1/2)$, and 
$J_m = \Big\lceil 2 \pi m^{1/2} \Big\rceil $.
Let us consider, for $\wt \alpha>0$ and $\De \in (0, \pi/3)$, the  two following events
\begin{gather*}
B^\A_m (\wt \alpha, \De) := \{ \A (\wt \alpha, \De )  \subset \wh L_m \}, 
\qquad
B^{\dist}_m (\De) := \{  \wh L_m \subset O_\De (\supp \mu)  \}.
\end{gather*}
Then,
\begin{gather*} 
\limsup_{m \to \infty} \frac{ \ln  \left(1- P \left( B^\A_m (\wt \alpha, \De) \right) \right) }{\ln^2 m} \ \le \
-\frac{3 \wt \alpha^2}{8 \pi^2 b_1^2} ,  \\
\limsup_{m \to \infty}
\frac{  \ln \left( 1- P \left( B_m^{\dist} (\De) \right) \right)}{ \ln^{2-2\de} m }
\ \le \
-\frac{3 }{8 \pi^2 b_1^2} .
\end{gather*}
\end{cor}

\vspace{3ex}

\noindent  Illya M. Karabash,
Institute of Applied Mathematics and Mechanics of NAS of Ukraine,
Dobrovolskogo st. 1, Slovyans'k 84100, Ukraine, and 
Humboldt Research Fellow at the University of Bonn,
Endenicher Allee 60, D-53115 Bonn, Germany.\\
E-mail:  i.m.karabash@gmail.com

\vspace{4ex}

\noindent  J\"urgen Prestin,
Institut f{\"u}r Mathematik, Universit{\"a}t zu L{\"u}beck, 
Ratzeburger Allee 160,\linebreak D-23562 L{\"u}beck, Germany.\\
E-mail: prestin@math.uni-luebeck.de

\end{document}